\documentclass[12pt]{article}
\usepackage{amssymb}
\usepackage{amsfonts}
\usepackage{amsmath}

\newtheorem{theorem}{Theorem}

\newtheorem{corollary}[theorem]{Corollary}

\newtheorem{definition}[theorem]{Definition}
\newtheorem{example}[theorem]{Example}

\newtheorem{lemma}[theorem]{Lemma}

\newtheorem{proposition}[theorem]{Proposition}
\newtheorem{remark}[theorem]{Remark}

\newenvironment{proof}[1][Proof]{\noindent\textbf{#1.} }{\ \rule{0.5em}{0.5em}}
\input{tcilatex}
\begin{document}

\author{Marek Galewski and Renata Wieteska}
\title{Multiple periodic solutions to a discrete $p\left( k\right) -$
Laplacian problem}
\date{}
\maketitle

\begin{abstract}
We investigate the existence of multiple periodic solutions to the
anisotropic discrete system. We apply the linking method and a new three
critical point theorem which we provide.
\end{abstract}

\section{Introduction}

Difference equations serve as mathematical models in diverse areas, such as
economy, biology, physics, mechanics, computer science, finance - see for
example \cite{agarwalBOOK}, \cite{elyadi}, \cite{lak}. Some of these models
are of independent interest since their mathematical structure allows for
obtaining new abstract tools. One of models arising in the study of elastic
mechanics is the $p\left( x\right) -$Laplacian. In this note we will study
its discrete counterpart, namely, we consider the existence of multiple $m-$%
periodic solutions to the following system%
\begin{equation}
\left\{ 
\begin{array}{l}
\Delta \left( |\Delta u(k-1)|^{p(k-1)-2}\Delta u(k-1)\right) +\lambda
f(k,u(k+1),u(k),u(k-1))=0,\bigskip \\ 
u(k+m)=u(k),\text{ \ \ \ \ \ \ \ \ \ \ \ \ \ \ \ \ \ \ \ \ \ \ \ \ \ \ \ \ \
\ \ \ \ \ \ \ \ \ \ \ \ \ \ \ \ \ \ \ \ \ \ \ \ \ \ \ \ \ \ \ \ }k\in 
\mathbb{Z}
.%
\end{array}%
\right.  \label{zad}
\end{equation}%
\qquad

The approach we apply is a variational one and concerns investigations of an
action functional in a suitable chosen space. Precisely speaking, we are
interested in finding at least three critical points to the action
functional connected with \textbf{(}\ref{zad}\textbf{)} using linking
arguments known in the literature, \cite{willem}, and a three critical point
theorem which we develop in this work as a generalization of the result from 
\cite{CABADSA2}, which does not apply to our case. The abstract tool which
we provide seems to apply also for continuous problems. Since the setting in
which we work is a discrete one, we have a different approach when compared
with continuous problems corresponding to a $p\left( x\right) -$Laplacian,
see for example \cite{hasto}.$\bigskip $

For the above system $\lambda >0$ is fixed and we will determine ranges for
parameter $\lambda $ corresponding to the existence of multiple solutions; $%
m\geq 2$ is a fixed natural number; $\left( \Delta u\right) (k-1)=u\left(
k\right) -u(k-1)$ stands for the forward difference operator; $u\left(
k\right) \in 
\mathbb{R}
^{n}$ for all $k\in 
\mathbb{Z}
$; $p:%
\mathbb{Z}
\rightarrow \lbrack 1,+\infty )$ is an $m-$periodic function, i.e. $p\left(
k+m\right) =p(k)$ for all $k\in 
\mathbb{Z}
;$ $f:%
\mathbb{Z}
\times 
\mathbb{R}
^{n}\times 
\mathbb{R}
^{n}\times 
\mathbb{R}
^{n}\rightarrow 
\mathbb{R}
^{n}$ is a continuous function $m-$periodic with respect to $k$; i.e. $%
f(k,u_{1},u_{2},u_{3})=f(k+m,u_{1},u_{2},u_{3})$ for all $%
(k,u_{_{1}},u_{_{2}},u_{3})\in 
\mathbb{Z}
\times 
\mathbb{R}
^{n}\times 
\mathbb{R}
^{n}\times 
\mathbb{R}
^{n}$; continuity means that for any fixed $k\in 
\mathbb{Z}
$ the function $f\left( k,\cdot ,\cdot ,\cdot \right) $ is continuous. We
underline that here $p$ need not satisfy $p\left( k\right) \geq 2\bigskip $
as is commonly assumed.

Several authors have investigated discrete BVPs with Dirichlet, periodic and
Neumann boundary conditions by the critical point theory. They applied
classical variational tools such as direct methods, the mountain geometry,
linking arguments, the degree theory. We refer to the following works far
from being exhaustive: \cite{ber1}, \cite{ber2}, \cite{uni1}, \cite{KoneOuro}%
, \cite{MRT}, \cite{bsehlik}, \cite{TianZeng}. Inspiration to our
investigations in this note lies in \cite{Liu}, where the discrete $p-$%
Laplacian is considered in \cite{CABADSA2}, where a new three critical point
theorem is developed, which is applicable to problems with the $p-$Laplacian
both discrete and continuos. We provide similar results as in the papers
mentioned, however, in the setting of the discrete $p\left( k\right) -$%
Laplacian. The discrete $p\left( k\right) -$Laplacian operator differs from
the classical discrete $p-$Laplacian and in this respect our investigations
are new.$\bigskip $

Continuous versions of problems like (\ref{zad}) are known to be
mathematical models of various phenomena arising in the study of elastic
mechanics, see \cite{B}, electrorheological fluids, see \cite{A}, or image
restoration, see \cite{C}. Variational continuous anisotropic problems were
started by Fan and Zhang in \cite{D} and later considered by many authors
and the use of many methods, see \cite{hasto} for an extensive survey of
such boundary value problems.$\bigskip $

Now, we provide some tools which are used throughout the paper.

\begin{definition}
Let $X$ be a normed space. We say that a functional $J:X\rightarrow \mathbb{R%
}$ is coercive if%
\begin{equation*}
\underset{\left\Vert u\right\Vert \rightarrow \infty }{\lim }J(u)=+\infty
\end{equation*}%
and anti-coercive if 
\begin{equation*}
\underset{\left\Vert u\right\Vert \rightarrow \infty }{\lim }J(u)=-\infty .
\end{equation*}
\end{definition}

\begin{definition}
Let $E$ be a Banach space. We say that a $C^{1}-$functional $J:E\rightarrow 
\mathbb{R}$ satisfies the Palais-Smale condition if every sequence $(u_{n})$
in $E$ for which $\{J(u_{n})\}$ is bounded and $J^{\prime
}(u_{n})\rightarrow 0$, has a convergent subsequence.
\end{definition}

\begin{proposition}
\cite{willem} \label{Liu}Let $E$ be a Banach space. Assume that $J\in
C^{1}(E,\mathbb{R})$ satisfies the Palais-Smale condition and is bounded
from below on $E$. Assume further that $J$ has a local linking at the origin 
$0$, namely, there exists a decomposition $E=Y\oplus W$, where $Y$ is a
finite dimensional subspace of $E,$ and a positive real number $\rho >0$ for
which%
\begin{equation*}
J(u)<J(0)\text{ for all }u\in Y\text{ }\ \text{with }0<\left\Vert
u\right\Vert \leq \rho
\end{equation*}%
and%
\begin{equation*}
J(u)\geq J(0)\text{ for all }u\in W\text{ }\ \text{with }\left\Vert
u\right\Vert \leq \rho .\text{ \ \ \ \ }
\end{equation*}%
Then $J$ has at least three critical points.$\bigskip $
\end{proposition}

The paper is organized as follows. Firstly, we provide some new three
critical point theorem applicable to anisotropic problems, then we give
variational formulation of problem under consideration. Existence and
multiplicity results are finally considered by linking arguments and our
multiplicity results.\textbf{\ }Finally some other recent three critical
point theorems are discussed with respect to applicability to our problem.
Examples are given throughout the text.

\section{Some generalization of the three critical point theorem}

\qquad In this section we follow \cite{CABADSA2} and \cite{GWAML} in order
to derive a type of the three critical points theorem, which would be
applicable in our case. The main result from \cite{GWAML} reads

\begin{theorem}[Main result]
\label{our}Let $(X,\left\Vert .\right\Vert )$ be a uniformly convex Banach
space with strictly convex dual space, $J\in C^{1}(X,%
\mathbb{R}
)$ be a functional with compact derivative, $\mu \in C^{1}(X,%
\mathbb{R}
_{+})$ be a convex coercive functional such that its derivative is an
operator $\mu ^{^{\prime }}:X\rightarrow X^{\ast }$ admitting a continuous
inverse. Let $\widetilde{x}\in X$ and $r>0$ be fixed. Assume that the
following conditions are satisfied:$\bigskip $\newline
\textbf{(B.1) }$\underset{\left\Vert x\right\Vert \rightarrow \infty }{\lim
\inf }\frac{J(x)}{\mu (x)}\geq 0;\bigskip $\newline
\textbf{(B.2)} $\underset{x\in X}{\inf }J(x)<\underset{\mu (x)\leq r}{\inf }%
J(x);\bigskip $\newline
\textbf{(B.3)} $\mu (\widetilde{x})<r$ and $J(\widetilde{x})<\underset{\mu
(x)=r}{\inf }J(x)$.$\bigskip $\newline
Then\ there exists a nonempty open set $A\subseteq (0,+\infty )$\ such that
for all $\lambda \in A$\ the functional $\mu +\lambda J$\ has at least three
critical points in $X$.
\end{theorem}

The Authors in \cite{CABADSA2} consider the above result in case $\mu \left(
x\right) =\left\Vert x\right\Vert ^{p}$ and they suggest to replace
condition \textbf{(B.1) }with\bigskip

\textbf{(B.4)}\ The functional $J$ is bounded from below on $X$, i.e. there
exists \ $C\in 
\mathbb{R}
$ such that $J(x)\geq C$ for every $x\in X.$\bigskip

They prove that in this case the three critical points result could also be
obtained. We follow their method in order to get the mentioned theorem.

\begin{theorem}
\label{our copy(1)}Let $(X,\left\Vert .\right\Vert )$ be a uniformly convex
Banach space with strictly convex dual space, $J\in C^{1}(X,%
\mathbb{R}
)$ be a functional with compact derivative, $\mu \in C^{1}(X,%
\mathbb{R}
_{+})$ be a convex coercive functional such that its derivative is an
operator $\mu ^{^{\prime }}:X\rightarrow X^{\ast }$ admitting a continuous
inverse. Let $\widetilde{x}\in X$ and $r>0$ be fixed. Assume that conditions 
\textbf{(B.2)-(B.4) }are satisfied. Then\ there exists a nonempty open set $%
A\subseteq (0,+\infty )$\ such \ that for all $\lambda \in A$\ the
functional $\mu +\lambda J$\ has at least three critical points in $X$.
\end{theorem}

\begin{proof}
Consider the functional $J_{C}\left( x\right) =J(x)-C$. We will show that
the functional $J_{C}$ satisfies conditions \textbf{(B.1)}-\textbf{(B.3)} of
Theorem \textbf{\ref{our}. }By \textbf{(B.4) }it follows that $J_{C}\left(
x\right) \geq 0$ for every $x\in X.$ Hence, \textbf{(B.1)} is satisfied with 
$J_{C}.$\newline
Since 
\begin{equation*}
\underset{x\in X}{\inf }J_{C}(x)=\underset{x\in X}{\inf }J(x)-C
\end{equation*}%
and%
\begin{equation*}
\underset{\mu (x)\leq r}{\inf }J_{C}(x)=\underset{\mu (x)\leq r}{\inf }%
J(x)-C,
\end{equation*}%
thus the functional $J_{C}$ satisfies conditions \textbf{(B.2)} and \textbf{%
(B.3)} of Theorem \textbf{\ref{our}.} Then there exists a nonempty open set $%
A\subseteq (0,+\infty )$\ such \ that for all $\lambda \in A$\ the
functional $\mu +\lambda J_{C}$\ has at least three critical points in $X$.
Since critical points of $\mu +\lambda J_{C}$ and $\mu +\lambda J$ coincide
we have the assertion.\bigskip \qquad
\end{proof}

Since a finite dimensional Hilbert space is a Banach space with strictly
convex dual space, we provide a result which is applicable for the problem
under consideration.

\begin{lemma}
\label{lemmaTHREEAPP}Let $(X,\left\Vert .\right\Vert )$ be a finite
dimensional Hilbert space, $J\in C^{1}(X,%
\mathbb{R}
)$ be bounded from below on $X$. Let $\mu \in C^{1}(X,%
\mathbb{R}
_{+})$ be a coercive functional whose derivative admits a continuous
inverse. Let $\widetilde{x}\in X$ and $r>0$ be fixed. Assume that conditions 
\textbf{(B.2)} and \textbf{(B.3)} are satisfied. Then\ there exists a
nonempty open set $A\subseteq (0,+\infty )$\ such \ that for all $\lambda
\in A$\ the functional $\mu +\lambda J$\ has at least three critical points
in $X$.
\end{lemma}

\section{Variational framework and auxiliary results}

\qquad Let $F:%
\mathbb{Z}
\times 
\mathbb{R}
^{n}\times 
\mathbb{R}
^{n}\rightarrow 
\mathbb{R}
$ be a $C^{1}-$ function; i.e. $F$ \ is continuous and has continuous
partial derivatives $F_{2}^{\prime }$, $F_{3}^{\prime }$ with respect to the
second and the third variable, respectively. Assume, apart from growth
conditions which will be given further, that $F$ has the following structure
properties:\bigskip \newline
\textbf{(A.1) }$f(k,u_{_{1}},u_{_{2}},u_{3})=F_{2}^{\prime
}(k-1,u_{2},u_{3})+F_{3}^{\prime }(k,u_{1},u_{2})$ for all $%
(k,u_{_{1}},u_{_{2}})\in 
\mathbb{Z}
\times 
\mathbb{R}
^{n}\times 
\mathbb{R}
^{n}$;\bigskip \newline
\textbf{(A.2)} $F(k,u_{1},u_{2})=F(k+m,u_{1},u_{2})$ for all $%
(k,u_{_{1}},u_{_{2}})\in 
\mathbb{Z}
\times 
\mathbb{R}
^{n}\times 
\mathbb{R}
^{n}$\bigskip ;\newline
\textbf{(A.3)} $F(k,0,0)=0$ for all \ $k\in $ $%
\mathbb{Z}
.$\bigskip

From now on, we will use the following notations%
\begin{equation*}
p^{-}=\underset{k\in \left[ 1,m\right] }{\min }p\left( k\right) ,\text{ \ \
\ \ \ \ \ \ \ }p^{+}=\underset{k\in \left[ 1,m\right] }{\max }p\left(
k\right) .
\end{equation*}

Define the space%
\begin{equation*}
H_{m}=\{u=\{u(k)\}_{k\in 
\mathbb{Z}
}:u(k)\in 
\mathbb{R}
^{n},\ u(k+m)=u(k),\text{ }k\in 
\mathbb{Z}
\},
\end{equation*}%
\bigskip which equipped with the Euclidean norm%
\begin{equation*}
\Vert u\Vert _{e}=\left( \sum_{k=1}^{m}|u(k)|^{2}\right) ^{1/2}
\end{equation*}%
becomes a Hilbert space.\bigskip

Put 
\begin{equation*}
W=\{u=\{u(k)\}_{k\in 
\mathbb{Z}
}:u(k)=a\in 
\mathbb{R}
^{n}\text{, }k\in 
\mathbb{Z}
\}\text{ \ and \ }Y=W^{\bot }\text{.}
\end{equation*}%
Thus $W$ consists of constant sequences and we have an orthogonal
decomposition%
\begin{equation*}
H_{m}=Y\oplus W.
\end{equation*}

With fixed $\lambda >0$ we define the action functional $J_{m}:H_{m}%
\rightarrow 
\mathbb{R}
$ corresponding to (\ref{zad}) by

\begin{equation*}
J_{m}\left( u\right) =\sum_{k=1}^{m}\left( \frac{1}{p(k-1)}\left\vert \Delta
u(k-1)\right\vert ^{p(k-1)}-\lambda F(k,u(k+1),u(k))\right) .\bigskip
\end{equation*}

\begin{lemma}
\label{Gateaux}Assume that conditions (\textbf{A.1)-(A.2)} hold and fix $%
\lambda >0$. Then the functional $J_{m}:H_{m}\longrightarrow 
\mathbb{R}
$ \textit{is continuously differentiable in the sense of G\^{a}teaux on }$%
H_{m}$\textit{. Moreover, }$u\in H_{m}$ \textit{is a critical point of }$%
J_{m}$ \textit{if and only if it satisfies}\ (\ref{zad}).
\end{lemma}

\begin{proof}
The continuity of $J_{m}$ is immediate. Let us take an arbitrary $u\in H_{m}$%
. Let $\varphi :%
\mathbb{R}
\longrightarrow 
\mathbb{R}
$ be given by the formula $\varphi (\varepsilon )=J_{m}\left( u+\varepsilon
h\right) ,$ where $h\in H_{m}$ is a fixed non-zero direction. Then%
\begin{equation*}
\begin{array}{l}
\varphi (\varepsilon )=\dsum\limits_{k=1}^{m}\frac{1}{p(k-1)}\left\vert
\Delta (u+\varepsilon h)(k-1)\right\vert ^{p(k-1)}\bigskip - \\ 
\lambda \dsum\limits_{k=1}^{m}F(k,(u+\varepsilon h)(k+1),(u+\varepsilon
h)(k)).%
\end{array}%
\end{equation*}%
Since $\varphi $ is continuously differentiable we have%
\begin{equation*}
\begin{array}{l}
\varphi ^{^{\prime }}(\varepsilon )=\dsum\limits_{k=1}^{m}\left\vert \Delta
(u+\varepsilon h)(k-1)\right\vert ^{p(k-1)-2}\Delta (u+\varepsilon
h)(k-1)\Delta h(k-1)-\bigskip \\ 
\lambda \dsum\limits_{k=1}^{m}(F_{2}^{^{\prime }}(k,(u+\varepsilon
h)(k+1),(u+\varepsilon h)(k))h(k+1)+\bigskip \\ 
F_{3}^{^{\prime }}(k,(u+\varepsilon h)(k+1),(u+\varepsilon
h)(k))h(k)).\bigskip%
\end{array}%
\end{equation*}%
Letting $\varepsilon =0$ and in view of \textbf{(A.2)} we get%
\begin{equation*}
\begin{array}{l}
\varphi ^{^{\prime }}(0)=\dsum\limits_{k=1}^{m}\left\vert \Delta
u(k-1)\right\vert ^{p(k-1)-2}\Delta u(k-1)\Delta h(k-1)-\bigskip \\ 
\lambda \dsum\limits_{k=1}^{m}\left( F_{2}^{^{\prime
}}(k-1,u(k),u(k-1))+F_{3}^{^{\prime }}(k,u(k+1),u(k))\right) h(k).\text{ \ \
\ \ \ \ }\bigskip%
\end{array}%
\end{equation*}%
Using Abel's summation by parts formula and owing to $m-$periodic conditions
and \textbf{(A.1)} we obtain\bigskip 
\begin{equation*}
\begin{array}{l}
\varphi ^{^{\prime }}(0)=-\dsum\limits_{k=1}^{m}(\Delta \left( |\Delta
u(k-1)|^{p(k-1)-2}\Delta u(k-1)\right) +\text{ \ \ \ \ \ \ \ \ \ \ \ \ \ \ \
\ }\bigskip \\ 
\lambda f(k,u(k+1),u(k),u(k-1)))h(k).%
\end{array}%
\end{equation*}%
We have shown that $J_{m}$ has a continuous G\^{a}teaux derivative. Letting
the derivative $0,$ we see that $u$ is a critical point of $J_{m}$ if and
only if it satisfies (\ref{zad}).
\end{proof}

\qquad Now, we prove some auxiliary results which we use later on.

\begin{lemma}
\label{CCC}\textbf{(C.1)}\textit{\ For every} $s>0$\textit{\ }%
\begin{equation*}
\dsum\limits_{k=1}^{m}\left\vert u(k)\right\vert ^{s}\leq m\left\Vert
u\right\Vert _{e}^{s}\text{ \textit{for all} }\mathit{u}\in H_{m}\text{.}%
\bigskip
\end{equation*}%
\textbf{(C.2) }\textit{For every }$s\geq 2$%
\begin{equation*}
\dsum\limits_{k=1}^{m}\left\vert u(k)\right\vert ^{s}\geq m^{\frac{2-s}{2}%
}\left\Vert u\right\Vert _{e}^{s}\text{ \textit{for all} }\mathit{u}\in
H_{m}.\bigskip
\end{equation*}%
\textit{\ }\textbf{(C.3)} \textit{For all }$u\in H_{m}$\textit{\ we have}%
\begin{equation*}
\sum_{k=1}^{m}|\Delta u(k-1)|^{p(k-1)}\leq m(2^{p^{+}}\Vert u\Vert
_{e}^{p^{+}}+1).\bigskip
\end{equation*}
\end{lemma}

\begin{proof}
To see \textbf{(C.1)} note that for all $k\in \lbrack 1,m]$ we have $\ \ \ \
\ \ $%
\begin{equation*}
\qquad \ |u(k)|^{2}\leq \dsum\limits_{i=1}^{m}|\Delta u(i)|^{2}.
\end{equation*}%
Thus$\ $%
\begin{equation*}
\qquad \ |u(k)|^{s}\leq \left( \left( \dsum\limits_{i=1}^{m}|\Delta
u(i)|^{2}\right) ^{\frac{1}{2}}\right) ^{s},
\end{equation*}%
which leads to%
\begin{equation*}
\qquad \dsum\limits_{k=1}^{m}\ |u(k)|^{s}\leq m\left\Vert u\right\Vert
_{e}^{s}.
\end{equation*}

Relation \textbf{(C.2)} follows immediately by H\"{o}lder's inequality.$%
\bigskip $

By \textbf{(C.1)} show we deduce that 
\begin{equation*}
\begin{array}{l}
\dsum\limits_{k=1}^{m}|\Delta u(k-1)|^{p(k-1)}\leq \bigskip \\ 
\dsum\limits_{\left\{ k\in \left[ 1,m\right] :\Delta u|(k-1)|\leq 1\right\}
}|\Delta u(k-1)|^{p^{-}}+\dsum\limits_{\left\{ k\in \left[ 1,m\right]
:|\Delta u(k-1)|>1\right\} }|\Delta u(k-1)|^{p+}\leq \bigskip \\ 
\dsum\limits_{k=1}^{m}|\Delta u(k-1)|^{p+}+\dsum\limits_{k=1}^{m}1\leq
2^{p^{+}}\dsum\limits_{k=1}^{m}\left\vert u(k)\right\vert ^{p^{+}}+m\leq
m\left( 2^{p+}\left\Vert u\right\Vert _{e}^{p^{+}}+1\right) .\bigskip%
\end{array}%
\end{equation*}%
Thus \textbf{(C.3) }holds.

The proof of Lemma \ref{CCC} is complete.
\end{proof}

\section{Multiple periodic solutions}

\subsection{Result by the linking method\label{sub-linking}}

\qquad In this section we investigate the existence $m-$periodic solutions
by applying Proposition \ref{Liu}.

Assume that $F$ satisfies additionally $\bigskip :$\newline
(\textbf{A.4) }There exist $m-$periodic functions $s,r:%
\mathbb{Z}
\rightarrow \lbrack 2,+\infty ),$ $\alpha _{1},\alpha _{2}:%
\mathbb{Z}
\rightarrow (0,+\infty )$ and a function $\alpha _{3}:%
\mathbb{Z}
\rightarrow 
\mathbb{R}
$\ for which%
\begin{equation*}
F(k,u_{_{1}},u_{_{2}})\geq \alpha _{1}(k)\left\vert u_{1}\right\vert
^{s(k)}+\alpha _{2}(k)\left\vert u_{2}\right\vert ^{r(k)}+\alpha
_{3}(k)\bigskip
\end{equation*}%
for\ all $k\in 
\mathbb{Z}
$ and all $u_{1},u_{2}\in 
\mathbb{R}
^{n}$\ \ such that $\left\vert u_{1}\right\vert ,\left\vert u_{2}\right\vert
\geq M$, where $M\geq 1$ is fixed and sufficiently large.

\textbf{(A.5) }There exists a constant $\eta >0$ such that 
\begin{equation*}
F(k,u_{1},u_{2})\geq 0\text{\ \ \ for\ all }k\in 
\mathbb{Z}
\text{ and all }u_{1},u_{2}\in 
\mathbb{R}
^{n}\text{\ \ for which \ }\left\vert u_{1}\right\vert +\left\vert
u_{2}\right\vert \leq 2\eta ;
\end{equation*}%
and at least one of the below conditions: \bigskip \newline
\textbf{(A.6.1) }$\underset{\left\vert u_{1}\right\vert +\left\vert
u_{2}\right\vert \rightarrow 0}{\lim }\frac{F(k,u_{1},u_{2})}{\left\vert
u_{1}\right\vert ^{s^{+}}+\left\vert u_{2}\right\vert ^{r-}}=0$ uniformly in 
$k\in 
\mathbb{Z}
;$\bigskip \newline
\textbf{(A.6.2) }$\underset{\left\vert u_{1}\right\vert +\left\vert
u_{2}\right\vert \rightarrow 0}{\lim }\frac{F(k,u_{1},u_{2})}{\left\vert
u_{1}\right\vert ^{s^{-}}+\left\vert u_{2}\right\vert ^{r^{+}}}=0$ uniformly
in $k\in 
\mathbb{Z}
;$\bigskip \newline
\textbf{(A.6.3)} $\underset{\left\vert u_{1}\right\vert +\left\vert
u_{2}\right\vert \rightarrow 0}{\lim }\frac{F(k,u_{1},u_{2})}{\left\vert
u_{1}\right\vert ^{s^{-}}+\left\vert u_{2}\right\vert ^{r^{-}}}=0$ uniformly
in $k\in 
\mathbb{Z}
,$

where

\begin{equation*}
s^{-}=\underset{k\in \left[ 1,m\right] }{\min }s\left( k\right) ,\text{ \ \
\ }s^{+}=\underset{k\in \left[ 1,m\right] }{\max }s\left( k\right) ,\text{ \ 
}r^{-}=\underset{k\in \left[ 1,m\right] }{\min }r\left( k\right) ,\text{ \ \
\ }r^{+}=\underset{k\in \left[ 1,m\right] }{\max }r\left( k\right) .\text{\ }
\end{equation*}

Note that both \textbf{(A.6.2)} and \textbf{(A.6.3) }imply \textbf{(A.6.4). }%
Now, we give some examples of nonlinear terms which can be considered by our
approach.

\begin{example}
\bigskip Let $m\geq 2$ be a fixed even natural\ number. Assume that 
\begin{equation*}
f(k,t_{1},t_{2},t_{3})=4t_{2}^{3}\left( 2+(-1)^{k}\left( \cos \left(
\left\vert t_{1}\right\vert ^{4}+\left\vert t_{2}\right\vert ^{4}\right)
-\cos \left( \left\vert t_{2}\right\vert ^{4}+\left\vert t_{3}\right\vert
^{4}\right) \right) \right) .
\end{equation*}%
for $(k,t_{1},t_{2},t_{3})\in 
\mathbb{Z}
\times 
\mathbb{R}
^{n}\times 
\mathbb{R}
^{n}\times 
\mathbb{R}
^{n}.$ Let us take the function $F$: $%
\mathbb{Z}
\times 
\mathbb{R}
^{n}\times 
\mathbb{R}
^{n}\rightarrow 
\mathbb{R}
$ given by 
\begin{equation*}
F(k,t_{_{1}},t_{_{2}})=\left\vert t_{1}\right\vert ^{4}+\left\vert
t_{2}\right\vert ^{4}+(-1)^{k}\sin \left( \left\vert t_{1}\right\vert
^{4}+\left\vert t_{2}\right\vert ^{4}\right)
\end{equation*}%
and functions $s,r:%
\mathbb{Z}
\rightarrow \lbrack 2,+\infty )$ such that%
\begin{equation*}
r(k)=s(k)=\left\{ 
\begin{array}{c}
4\text{ \ \ \ \ \ \ \ \ \ \ \ \ \ \ \ \ \ \ \ \ \ \ \ \ \ \ \ \ \ \ \ \ \
for }k=2l, \\ 
2\ \text{\ \ \ \ \ \ \ \ \ \ \ \ \ \ \ \ \ \ for }k=2l+1;l\in 
\mathbb{Z}
.%
\end{array}%
\right.
\end{equation*}%
With such defined functions we see that conditions \textbf{(A.1)}-\textbf{%
(A.5)} and \textbf{(A.6.3) }are fulfilled with $M=1$, $\eta \in (0,\frac{1}{2%
}],$ $\alpha _{1}(k)=\alpha _{2}(k)=1,$ $\alpha _{3}(k)=-1$ for all $k\in 
\mathbb{Z}
,$ but conditions \textbf{(A.6.1), (A.6.2)} are not.
\end{example}

\begin{example}
Let $m\geq 2$ be a fixed natural\ number. Assume that 
\begin{equation*}
f(k,t_{1},t_{2},t_{3})=4t_{2}^{3}\left( \cos ^{2}\left( \frac{k-1}{m}\pi
\right) t_{3}^{4}+\cos ^{2}\left( \frac{k}{m}\pi \right) t_{1}^{4}\right) .
\end{equation*}%
for $(k,t_{1},t_{2},t_{3})\in 
\mathbb{Z}
\times 
\mathbb{R}
^{n}\times 
\mathbb{R}
^{n}\times 
\mathbb{R}
^{n}.$ Let us take the function $F$: $%
\mathbb{Z}
\times 
\mathbb{R}
^{n}\times 
\mathbb{R}
^{n}\rightarrow 
\mathbb{R}
$ given by 
\begin{equation*}
F(k,t_{_{1}},t_{_{2}})=\cos ^{2}\left( \frac{k}{m}\pi \right) \left\vert
t_{1}t_{2}\right\vert ^{4}
\end{equation*}%
and functions $s,r:%
\mathbb{Z}
\rightarrow \lbrack 2,+\infty )$ such that%
\begin{equation*}
r(k)=s(k)=\sin (\frac{k}{m}\pi )+3.
\end{equation*}%
With such defined functions conditions \textbf{(A.1)}-\textbf{(A.5)} and 
\textbf{(A.6.1)}-\textbf{(A.6.3)} are fulfilled with $M=\sqrt[4]{2},$ $%
\alpha _{1}(k)=\alpha _{2}(k)=\cos ^{2}\left( \frac{k}{m}\pi \right) ,$ $%
\alpha _{3}(k)=0$ for all $k\in 
\mathbb{Z}
.$
\end{example}

Set%
\begin{equation*}
\alpha _{i}^{-}=\underset{k\in \left[ 1,m\right] }{\min }\alpha _{i}\left(
k\right) \text{\ \ \ \ for\ \ \ }i=1,2,3\bigskip .
\end{equation*}

\begin{lemma}
\label{anti-coercive}Suppose that conditions \textbf{(A.1)-(A.2) }are
satisfied. Assume that \textbf{(A.4)} holds with either $s^{-}>p^{+}$ or $\
r^{-}>p^{+}$. Then the functional $J_{m}$ is anti-coercive on $H_{m}$ for
all $\lambda >0$.
\end{lemma}

\begin{proof}
Using \textbf{(A.4)}, \textbf{(C.2)} \ and\textbf{\ (C.3)} we obtain%
\begin{equation*}
\begin{array}{l}
J_{m}\left( u\right) \leq \frac{1}{p^{-}}\bigskip m\left( 2^{p+}\left\Vert
u\right\Vert _{e}^{p^{+}}+1\right) - \\ 
\lambda \dsum\limits_{k=1}^{m}\left( \alpha _{1}(k)\left\vert
u(k+1)\right\vert ^{s(k)}+\alpha _{2}(k)\left\vert u(k)\right\vert
^{r(k)}+\alpha _{3}(k)\right) \leq \text{ \ \ \ \ \ \ \ \ \ \ \ \ \ \ \ \ \
\ \ \ }%
\end{array}%
\end{equation*}%
\begin{equation*}
\begin{array}{l}
\frac{1}{p^{-}}m\left( 2^{p+}\left\Vert u\right\Vert _{e}^{p^{+}}+1\right)
-\lambda \alpha _{1}^{-}\dsum\limits_{k=1}^{m}\left\vert u(k)\right\vert
^{s^{-}}-\lambda \alpha _{2}^{-}\dsum\limits_{k=1}^{m}\left\vert
u(k)\right\vert ^{r^{-}}-\lambda m\alpha _{3}^{-}\leq \bigskip \\ 
\frac{1}{p^{-}}\bigskip 2^{p+}m\left\Vert u\right\Vert _{e}^{p^{+}}-\lambda
\alpha _{1}^{-}m^{\frac{2-s^{-}}{2}}\left\Vert u\right\Vert
_{e}^{s^{-}}-\lambda \alpha _{2}^{-}m^{\frac{2-r^{-}}{2}}\left\Vert
u\right\Vert _{e}^{r^{-}}+\left( \frac{1}{p^{-}}-\lambda \alpha
_{3}^{-}\right) m.%
\end{array}%
\end{equation*}%
Since $s^{-}>p^{+}$ or $r^{-}>p^{+}$, so $J_{m}$ is anti-coercive on $%
H_{m}.\bigskip $
\end{proof}

Put

\begin{equation*}
\lambda _{1}=\frac{2^{p^{+}}m^{\frac{p^{+}}{2}}}{p^{-}\alpha _{1}^{-}};\text{
\ \ }\lambda _{2}=\frac{2^{p^{+}}m^{\frac{p^{+}}{2}}}{p^{-}\alpha _{2}^{-}};%
\text{ \ \ }\lambda _{3}=\frac{2^{p^{+}}m^{\frac{p^{+}}{2}}}{p^{-}\left(
\alpha _{1}^{-}+\alpha _{2}^{-}\right) }.
\end{equation*}

\begin{lemma}
Suppose that conditions \textbf{(A.1)-(A.2) }and\textbf{\ (A.4) }are
satisfied. The following assertions are true:\newline
(a) If $s^{-}=p^{+}$ and $r^{-}<p^{+}$, then the functional $J_{m}$ is
anti-coercive on $H_{m}$ for any $\lambda \in \left( \lambda _{1},+\infty
\right) $;\newline
(b) If $r^{-}=p^{+}$ and $s^{-}<p^{+}$, then the functional $J_{m}$ is
anti-coercive on $H_{m}$ for any $\lambda \in \left( \lambda _{2},+\infty
\right) $; \newline
(c) If $r^{-}=s^{-}=p^{+}$, then the functional $J_{m}$ is anti-coercive on $%
H_{m}$ for any $\lambda \in \left( \lambda _{3},+\infty \right) $.
\end{lemma}

\begin{proof}
Assume that $s^{-}=p^{+}$ and $r^{-}<p^{+}.$ Let $\lambda \in \left( \lambda
_{1},+\infty \right) .$ Arguing as in the proof of Lemma \ref{anti-coercive}
we obtain%
\begin{equation*}
J_{m}\left( u\right) \leq \left( \frac{1}{p^{-}}2^{p^{+}}m-\lambda \alpha
_{1}^{-}m^{\frac{2-p^{+}}{2}}\right) \Vert u\Vert _{e}^{p^{+}}-\lambda
\alpha _{2}^{-}m^{\frac{2-r^{-}}{2}}\Vert u\Vert _{e}^{r^{-}}+\left( \frac{1%
}{p^{-}}-\lambda \alpha _{3}^{-}\right) m.
\end{equation*}%
Hence assertion (a) holds. In the remaining cases we proceed as in the above.%
$\bigskip $
\end{proof}

\begin{remark}
\label{ksip}It easy to verify that 
\begin{equation*}
\left\Vert u\right\Vert _{p^{+}}=\left( \dsum\limits_{k=1}^{m}\left\vert
\Delta u(k-1)\right\vert ^{p^{+}}\right) ^{\frac{1}{p^{+}}}
\end{equation*}%
is also a norm on $Y$, while it is obviously not a norm on $H_{m}$. Since
all norms on $Y$ are equivalent, therefore there exists a constant\textit{\ }%
$\xi >0$\textit{\ }such that\textit{\ } 
\begin{equation}
\dsum\limits_{k=1}^{m}|\Delta u(k-1)|^{p^{+}}\geq \xi \left\Vert
u\right\Vert _{e}^{p^{+}}.  \label{ksi}
\end{equation}
\end{remark}

Now, We are able to formulate our main result in two cases, depending on
relation between functions $r$, $s$ and $p$.$\bigskip $

\textit{Case I. }$s^{-}>p^{+}.\bigskip $

\begin{theorem}
Let $s^{-}>p^{+}$ and $\lambda >0$ be fixed. Assume that conditions \textbf{%
(A.1)-(A.5)} are satisfied\textbf{\ }and that at least one of the following
conditions\ holds:\newline
(a) \textbf{(A.6.2)} with $s^{-}\leq r^{+};$ \bigskip \newline
(b) \textbf{(A.6.3)} with $s^{-}\leq r^{-}.$\bigskip \newline
Then problem (\ref{zad}) has at least three $m$-periodic solutions, two of
which are nontrivial.
\end{theorem}

\begin{proof}
Assume that \textbf{(A.1)-(A.5) }with\textbf{\ }(a) hold. Choose a positive
real number $\varepsilon $ satisfying 
\begin{equation*}
\varepsilon <\frac{\xi }{2\lambda mp^{+}},
\end{equation*}%
where $\xi >0$ is a constant defined in Remark \ref{ksip}$.$ By \textbf{(A.5)%
}, \textbf{(A.6.2)} there exists $\rho \in (0,\rho _{0}),$ where $\rho
_{0}=\min \left\{ 1,\eta \right\} $, such that 
\begin{equation}
F(k,u_{1},u_{2})\leq \varepsilon (\left\vert u_{1}\right\vert
^{s^{-}}+\left\vert u_{2}\right\vert ^{r^{+}})\text{ \ \ \ for \ \ }%
\left\vert u_{1}\right\vert +\left\vert u_{2}\right\vert \leq 2\rho \text{.}
\label{eps}
\end{equation}%
If $u\in Y$ with $0<\left\Vert u\right\Vert _{e}\leq \rho $ then $\left\vert
u(k)\right\vert \leq \rho $ for all $k\in 
\mathbb{Z}
.$ By \textbf{(}\ref{ksi}\textbf{)}, (\ref{eps}) and \textbf{(C.1)} we obtain%
\begin{equation*}
\begin{array}{l}
J_{m}\left( u\right) \geq \frac{1}{p^{+}}\bigskip \xi \left\Vert
u\right\Vert _{e}^{p^{+}}-\lambda \varepsilon
\dsum\limits_{k=1}^{m}(\left\vert u(k+1)\right\vert ^{s^{-}}+\left\vert
u(k)\right\vert ^{r^{+}})\geq \\ 
\frac{1}{p^{+}}\xi \left\Vert u\right\Vert _{e}^{p^{+}}-\lambda \varepsilon
\dsum\limits_{k=1}^{m}(\left\vert u(k)\right\vert ^{s^{-}}+\left\vert
u(k)\right\vert ^{r^{+}})\geq \bigskip \\ 
\frac{1}{p^{+}}\xi \left\Vert u\right\Vert _{e}^{p^{+}}-\lambda \varepsilon
m(\left\Vert u\right\Vert _{e}^{s^{-}}+\left\Vert u\right\Vert
_{e}^{r^{+}})\geq \Vert u\Vert _{e}^{s-}\left( \frac{1}{p^{+}}\bigskip \xi
-2\lambda \varepsilon m\right) >0.%
\end{array}%
\end{equation*}%
Thus, by above and \textbf{(A.3)} we obtain%
\begin{equation*}
J_{m}(u)>J_{m}(0)\text{ for all }u\in Y\text{ }\ \text{with }0<\left\Vert
u\right\Vert _{e}\leq \rho .
\end{equation*}%
Notice that for every $u\in W$ we have $\Delta u(k-1)=0$ for all $k\in 
\mathbb{Z}
,$ so%
\begin{equation*}
J_{m}\left( u\right) =-\lambda \dsum\limits_{k=1}^{m}F(k,u(k+1),u(k)).
\end{equation*}%
If $u\in W$ with $\left\Vert u\right\Vert _{e}\leq \rho $ then 
\begin{equation*}
\text{\ }\left\vert u(k+1)\right\vert +\left\vert u(k)\right\vert \leq 2\eta
\end{equation*}%
for all $k\in 
\mathbb{Z}
.$ Thus, by \textbf{(A.5)} and \textbf{(A.3)} it follows that 
\begin{equation*}
J_{m}(u)\leq J_{m}(0)\text{ for all }u\in W\text{ }\ \text{with }\left\Vert
u\right\Vert _{e}\leq \rho .
\end{equation*}%
Let $\Phi _{m}=-J_{m}.$ We see that $\Phi _{m}$ has a local linking at the
origin $0$ with respect to the decomposition $H_{m}=Y\oplus W.$

By Lemma \ref{anti-coercive} we deduce that $\Phi _{m}$ satisfies the
Palais-Smale condition. Moreover, $\Phi _{m}$ is bounded from below, since
as coercive and continuous has a minimizer. We have shown that the
assumptions of Lemma \ref{Liu} are satisfied, so $\Phi _{m}$ has at least
three critical points, two of them are nonzero critical points. By Lemma \ref%
{Gateaux}\ these are nontrivial $m-$periodic solutions of problem (\ref{zad}%
). In case (b) we follow analogously.$\bigskip $
\end{proof}

For $s^{-}=p^{+}$ with $r^{-}<p^{+}$ we can observe that

\begin{corollary}
Let $s^{-}=p^{+}$ with $r^{-}<p^{+}$ and let $\lambda \in (\lambda
_{1},+\infty ).$ Assume that conditions \textbf{(A.1)-(A.5)} and condition 
\textbf{(A.6.2) }with $s^{-}\leq r^{+}$ are satisfied. Then problem (\ref%
{zad}) has at least three $m-$periodic solutions, two of which are
nontrivial.$\bigskip $
\end{corollary}

\textit{Case II. }$r^{-}>p^{+}.\bigskip $

\begin{theorem}
Let $r^{-}>p^{+}$ and $\lambda >0$ be fixed. Assume that conditions \textbf{%
(A.1)-(A.5)} are satisfied and at least one of the following conditions
holds\bigskip :\newline
(a) \textbf{(A.6.1)} with $r^{-}\leq s^{+}$; \bigskip \newline
(b) \bigskip \textbf{(A.6.3)} with $r^{-}\leq s^{-}.$\newline
Then problem (\ref{zad}) has at least three $m$-periodic solutions, two of
which are nontrivial.\bigskip
\end{theorem}

For $r^{-}=p^{+}$ with $s^{-}<p^{+}$ we can observe that

\begin{corollary}
Let $r^{-}=p^{+}$ with $s^{-}<p^{+}$ and let $\lambda \in (\lambda
_{2},+\infty ).$ Assume that conditions \textbf{(A.1)-(A.5) }and condition%
\textbf{\ (A.6.1)} with $r^{-}\leq s^{+}$ are satisfied.\newline
Then problem (\ref{zad}) has at least three $m$-periodic solutions, two of
which are nontrivial.
\end{corollary}

In particular,$\ $ if $s^{-}=r^{-}=p^{+}$ we obtain

\begin{corollary}
Let $s^{-}=r^{-}=p^{+}$ and let $\lambda \in (\lambda _{3},+\infty ).$
Assume that conditions \textbf{(A.1)-(A.5)} and \textbf{(A.6.3) }are
satisfied. Then problem (\ref{zad}) has at least three $m-$periodic
solutions, two of which are nontrivial.$\bigskip $
\end{corollary}

\subsection{Result by the three critical point theorem}

\qquad In this section we provide multiplicity results for problem (\ref{zad}%
) using Lemma \ref{lemmaTHREEAPP}. Assume that $F$ apart from structure
conditions \textbf{(A.1)-(A.3)} has the following properties:\bigskip 
\newline
\textbf{(A.7) }\textit{There exists\ a constant }$C\in 
\mathbb{R}
$\textit{\ such that }%
\begin{equation*}
F\left( k,u_{1},u_{2}\right) \leq C\text{ for all }\left(
k,u_{1},u_{2}\right) \in 
\mathbb{Z}
\times \mathbb{%
\mathbb{R}
}^{n}\times \mathbb{%
\mathbb{R}
}^{n}\bigskip ;
\end{equation*}%
\textbf{(A.8)} \textit{There exists a number }$\rho _{1}>0$\textit{\ such
that}%
\begin{equation*}
F\left( k,u_{1},u_{2}\right) <0
\end{equation*}%
\textit{\ for all }$k\in 
\mathbb{Z}
$\textit{\ and all }$u_{1},u_{2}\in 
\mathbb{R}
^{n}$\textit{\ for which }$0<\left\vert u_{1}\right\vert \leq \rho _{1}$%
\textit{, }$0<\left\vert u_{2}\right\vert \leq \rho _{1};\bigskip $\newline
\textbf{(A.9) }\textit{There exist numbers }$\rho _{3},\rho _{2}$\textit{\
such that }$\rho _{3}\geq \rho _{2}>\rho _{1}$\textit{\ and}%
\begin{equation*}
F\left( k,u_{1},u_{2}\right) >0
\end{equation*}%
\textit{\ for all }$k\in 
\mathbb{Z}
$\textit{\ and all }$u_{1},u_{2}\in 
\mathbb{R}
^{n}$\textit{\ for which }$\rho _{2}<\left\vert u_{1}\right\vert \leq \rho
_{3}$\textit{, }$\rho _{2}<\left\vert u_{2}\right\vert \leq \rho
_{3}.\bigskip $

\begin{theorem}
\label{C2C4}Assume that conditions \textbf{(A.1)-(A.3) }and\textbf{\
(A.7)-(A.9)} hold. Then\ there exists a nonempty open set $A\subseteq
(0,+\infty )$\ such \ that for all $\lambda \in A$\ problem (\ref{zad}) has
at least three solutions in $Y$, two of which are necessarily non-zero.
\end{theorem}

\begin{proof}
We will consider problem (\ref{zad}) in $Y$ since on $H_{m}$ functional $\mu 
$ given by 
\begin{equation*}
\mu (u)=\sum_{k=1}^{m}\frac{1}{p(k-1)}|\Delta u(k-1)|^{p(k-1)}
\end{equation*}%
is not coercive. Thus $\mu :Y\rightarrow 
\mathbb{R}
$ and we define $J:Y\rightarrow 
\mathbb{R}
$ by 
\begin{equation*}
J(u)=-\sum_{k=1}^{m}F(k,u(k+1),u(k)).
\end{equation*}%
Then $J_{m}(u)=\mu (u)+\lambda J(u)$ on $Y$. We will show that $\mu $ is
coercive on $Y$. Let us take $u\in Y$ \ with$\ \left\Vert u\right\Vert
_{p^{-}}^{p^{-}}>1$. Recalling Remark \ref{ksip} we have 
\begin{equation*}
\begin{array}{l}
\mu (u)\geq \frac{1}{p^{+}}\left( \dsum\limits_{\left\{ k\in \left[ 1,m%
\right] :|\Delta u(k-1)|\leq 1\right\} }|\Delta u(k-1)|^{p(k-1)}\right.
+\bigskip \\ 
\left. \dsum\limits_{\left\{ k\in \left[ 1,m\right] :|\Delta
u(k-1)|>1\right\} }|\Delta u(k-1)|^{p(k-1)}\right) \geq \bigskip \bigskip \\ 
\frac{1}{p^{+}}\left( \dsum\limits_{k=1}^{m}|\Delta
u(k-1)|^{p-}-\dsum\limits_{k=1}^{m}1\right) \geq \frac{1}{p^{+}}\left\Vert
u\right\Vert _{p^{-}}^{p^{-}}-\frac{1}{p^{+}}m.\bigskip%
\end{array}%
\end{equation*}%
Thus we see that $\mu $ is coercive on $Y$.$\bigskip $

By \textbf{(A.7)} we find that $J$ is bounded from below on $Y$. From 
\textbf{(A.9)} it follows that 
\begin{equation*}
-F\left( k,u(k+1),u(k)\right) <0\text{ }
\end{equation*}%
for all $u\in Y$ such that $\rho _{2}<\left\vert u(k)\right\vert \leq \rho
_{3},$ for all $k\in $ $%
\mathbb{Z}
.$ This means, that there exists a point $u\in Y$ such that $J(u)<0.$ Hence%
\begin{equation*}
\underset{u\in Y}{\inf }J(u)<0.
\end{equation*}%
For all $u\in Y$ satisfying \textbf{(A.8)} it is clear that $\left\vert
u(k)\right\vert \leq \rho _{1}$ \ for all $k\in 
\mathbb{Z}
.$ So $\left\vert \Delta u(k-1)\right\vert \leq 2\rho _{1}$ \ for all $k\in 
\mathbb{Z}
$ and consequently%
\begin{equation*}
\frac{1}{p(k-1)}\left\vert \Delta u(k-1)\right\vert ^{p(k-1)}\leq \frac{1}{%
p(k-1)}\left( 2\rho _{1}\right) ^{p(k-1)}\text{ \ for all }k\in 
\mathbb{Z}
.
\end{equation*}%
Thus $\mu (u)\leq r_{2},$ where $r_{2}=\dsum\limits_{k=1}^{m}\frac{1}{p(k-1)}%
\left( 2\rho _{1}\right) ^{p(k-1)}.$ Since $\mu $ is continuous, coercive,
non-negative and $\mu (0)=0$, we get $r\in (0,r_{2})$ such that\ $J(u)\geq 0$%
\ for $\mu (u)\leq r$, by \textbf{(A.3) }and \textbf{(A.8)}$.$ Therefore 
\textbf{(B.2)}\ is satisfied. Now, putting $\widetilde{u}=0$ we observe that 
\begin{equation*}
0=\mu (0)<r\text{ and }0=J\left( 0\right) <\underset{\mu (u)=r}{\inf }J(u).
\end{equation*}%
Hence, condition \textbf{(B.3)}\ is satisfied. Thus, by Lemma \ref%
{lemmaTHREEAPP} we see that there exists a nonempty open set $A\subseteq
(0,+\infty )$\ such \ that for all $\lambda \in A$\ the functional $J_{m}$
has at least three critical points on $Y$. Since by Lemma \ref{Gateaux}
critical points of $J_{m}$ are solutions of problem (\ref{zad}), we get the
assertion. $\bigskip $
\end{proof}

Since we obtain solutions in $Y$ we know that these are not constant
functions.

Now, we give an example to illustrate Theorem \ref{C2C4}.

\begin{example}
Let $m\geq 2$ be a fixed natural number. Let\ us consider the function $F:%
\mathbb{Z}
\times 
\mathbb{R}
^{n}\times 
\mathbb{R}
^{n}\rightarrow 
\mathbb{R}
$ given by the formula%
\begin{equation*}
F(k,u_{1},u_{2})=-\sin (u_{1}^{2}+u_{2}^{2})\left\vert \sin \left( \frac{k}{m%
}\pi \right) \right\vert
\end{equation*}%
which satisfies conditions \textbf{(A.1)}-\textbf{(A.3)} and \textbf{(A.7)}-%
\textbf{(A.9)}.
\end{example}

\section{Final comments}

There are other abstract theorems which pertain to the existence of three
critical points requiring different sets of assumptions, see for example 
\cite{bonano-candito}, Theorem 2.1, \cite{bonano-candito2}, Theorem 3.3 and 
\cite{bonano-marano}, Theorem 2.6. However the applicability of these
results due to the structure of discrete $p\left( k\right) -$Laplacian being
different from this of $p-$Laplacian seems to be much more difficult.

Much more suitable seems the multiplicity Theorem 2.3 from \cite%
{bonano-candito}. We use notation from the proof of Theorem \ref{C2C4}.

Put%
\begin{equation*}
\lambda ^{\ast }=\left( \underset{\left( r>\inf_{X}\mu \right) }{\inf }\text{
}\underset{u\in \mu ^{-1}\left( (-\infty ,r)\right) }{\inf }\frac{\left( 
\underset{u\in \mu ^{-1}\left( (-\infty ,r)\right) }{\sup }J(u)\right) -J(u)%
}{r-\mu (u)}\right) ^{-1},
\end{equation*}%
where we read $\frac{1}{0}=+\infty $ if this case occurs.

\begin{theorem}
\label{bon-can}\cite{bonano-candito} Let $(X,\left\Vert .\right\Vert )$ be a
finite dimensional Banach space and $\mu ,$ $J\in C^{1}(X,%
\mathbb{R}
)$ with $\mu $ coercive. Assume that the functional $\mu -\lambda J$ is
anticoercive on $X$ for all $\lambda \in (0,\lambda ^{\ast })$. Then, for
all $\lambda \in (0,\lambda ^{\ast })$ \ the functional $\mu -\lambda J$
admits at least three distinct critical points.
\end{theorem}

Let us compare results from Subsection 4.2 the above result from \cite%
{bonano-candito}. Firstly, we seem to have much more precise eigenvalue
intervals. Next, these eigenvalue intervals are relatively easy to be
determined. We however underline that in the application of Theorem 2.3 from 
\cite{bonano-candito} no condition around $0$ is required but we are
confined to the case when the action functional is anti-coercive for all $%
\lambda >0$ since it is rather difficult to calculate $\lambda ^{\ast }$
directly. So both approaches have their own advantages and
disadvantages.\bigskip

The multiplicity results obtain by using Theorem \ref{bon-can} reads

\begin{theorem}
Suppose that conditions \textbf{(A.1)-(A.2) }are satisfied. Assume that 
\textbf{(A.4)} holds with either $s^{-}>p^{+}$ or $\ r^{-}>p^{+}.$ Then for
all $\lambda \in \left( 0,\lambda ^{\ast }\right) $ problem (\ref{zad}) has
at least three distinct solutions in $Y$.
\end{theorem}

\begin{tabular}{l}
Marek Galewski, Renata Wieteska \\ 
Institute of Mathematics, \\ 
Technical University of Lodz, \\ 
Wolczanska 215, 90-924 Lodz, Poland, \\ 
marek.galewski@p.lodz.pl, renata.wieteska@p.lodz.pl%
\end{tabular}


\begin{thebibliography}{99}
\bibitem{agarwalBOOK} R. P. Agarwal, Difference Equations and Inequalities,
Marcel Dekker, New York, 1992.

\bibitem{ber1} C. Bereanu, P. Jebelean, C. \c{S}erban, Periodic and Neumann
problems for discrete $p(\cdot )-$Laplacian, J. Math. Anal. Appl. 399
(2013), 75--87

\bibitem{ber2} C. Bereanu, P. Jebelean, C. \c{S}erban, Ground state and
mountain pass solutions for discrete $p(\cdot )-$Laplacian, Bound. Value
Probl. 2012 (104) (2012) 1-13.

\bibitem{bonano-candito} G. Bonanno, P. Candito, Nonlinear difference
equations investigated via critical point methods, Nonlinear Anal., Theory
Methods Appl. A, 70 (2009), No. 9, 3180-3186.

\bibitem{bonano-candito2} G. Bonanno, P. Candito, Non-differentiable
functionals and applications to elliptic problems with discontinuous
nonlinearities, J. Differ. Eqs 244 (2008) 3031-3059.

\bibitem{bonano-marano} G. Bonanno, S.A. Marano, On the structure of the
critical set of non-differentiable functions with a weak compactness
condition, Appl. Anal. 89 (2010) 1-10.

\bibitem{CABADSA2} A. Cabada, S. Tersian, Multiplicity of solutions of a two
point boundary value problem for a fourth-order equation, Appl. Math.
Comput. 219 (2013) 5261--5267.

\bibitem{C} Y. Chen, S. Levine and M. Rao, Variable exponent, linear growth
functionals in image processing, SIAM J. Appl. Math. 66 (2006), No. 4,
1383-1406.

\bibitem{elyadi} S. N. Elaydi, An Introduction to Difference Equations,
Springer-Verlag, New York, 1999.

\bibitem{D} X. L. Fan , H. Zhang, Existence of Solutions for $p(x)-$Lapacian
Dirichlet Problem, Nonlinear Anal., Theory Methods Appl. 52, No. 8, A,
1843-1852 (2003).

\bibitem{GWAML} M. Galewski, R. Wieteska, A note on the multiplicity of
solutions to anisotropic discrete BVP's, Appl. Math. Lett. 26 (2013)
524--529.

\bibitem{uni1} A. Guiro, I. Nyanquini, S. Ouaro, On the solvability of
discrete nonlinear Neumann problems involving the $p(x)-$Laplacian, Adv.
Difference Equ. 011, 2011:32.

\bibitem{hasto} P. Harjulehto, P. H\"{a}st\"{o}, U. V. Le and M. Nuortio,
Overview of differential equations with non-standard growth, Nonlinear Anal.
72 (2010), 4551-4574.

\bibitem{KoneOuro} B. Kone. S. Ouaro, Weak solutions for anisotropic
discrete boundary value problems, J. Difference Equ. Appl., 17 (2011), No.
10, 1537-1547.

\bibitem{lak} V. Lakshmikantham and D. Trigiante, Theory of Difference
Equations: Numerical Methods and Applications, Academic Press, New York,
1988.

\bibitem{Liu} S. Liu, Multiple periodic solutions for nonlinear difference
systems involving the p-Laplacian, J. Difference Equ. Appl., Vol. 17, No. 11
(2011) 1591-1598.

\bibitem{MRT} M. Mih\v{a}ilescu, V. R\v{a}dulescu, S. Tersian, Eigenvalue
problems for anisotropic discrete boundary value problems. J. Difference
Equ. Appl. 15 (2009), No. 6, 557--567.

\bibitem{A} M. R\r{u}\v{z}i\v{c}ka, Electrorheological fluids: Modelling and
Mathematical Theory, in: Lecture Notes in Mathematics, vol. 1748,
Springer-Verlag, Berlin, 2000.

\bibitem{bsehlik} P. Stehl\'{\i}k, On variational methods for periodic
discrete problems, J. Difference Equ. Appl. 14 (2008), No. 3, 259-273.

\bibitem{TianZeng} Y. Tian, Z. Du and W. Ge, Existence results for discrete
Sturm-Liouville problem via variational methods, J. Difference Equ. Appl. 13
(2007), No. 6, 467--478.

\bibitem{willem} M. Willem, Minimax Theorem, Birkh\"{a}user, 1996.

\bibitem{B} V. V. Zhikov, Averaging of functionals of the calculus of
variations and elasticity theory, Math. USSR Izv. 29 (1987), 33-66.
\end{thebibliography}
\end{document}